\documentclass[11pt]{amsart}

\usepackage{amsfonts}
\usepackage{graphicx,amssymb,amsfonts,amsmath,amscd}
\usepackage[all,ps,cmtip]{xy}
\usepackage{amscd}

\usepackage{mathrsfs}
\let\mathcal\mathscr

\usepackage{hyperref}

\def\Alb{\mathop{\rm Alb}\nolimits}

\def\dim{\mathop{\rm dim}\nolimits}

\def\rank{\mathop{\rm rank}\nolimits}

\def\Aut{\mathop{\rm Aut}\nolimits}
\def\sym{\mathbf{S}}
\def\Sym{\mathbf{S}}

\def\llra{\hbox to 12mm{\tofill}}

\def\Sp{\mathop{\rm Sp}\nolimits}
\def\fI{\mathop{\rm I}\nolimits}
\def\Hdg{\mathop{\rm Hdg}\nolimits}

\def\s{{\sigma}}

\def\min{\mathop{\rm min}\nolimits}

\def\cI{{\mathcal I}}
\def\cJ{{\mathcal J}}
\def\cK{{\mathcal K}}

\def\cO{{\mathcal O}}

\def\cZ{{\mathcal Z}}
\def\1Y{{Y'}}

\newenvironment{proofth}{\trivlist \item[\hskip\labelsep{\sc
Proof of the Proposition \ref{0.3}.}]\rm}{\hbox
to.1pt{\hss}\hfill$\square$\bigskip\endtrivlist}

\newtheorem{theo}{Theorem}[section]
\newtheorem{prop}[theo]{Proposition}
\newtheorem{lemm}[theo]{Lemma}
\newtheorem{coro}[theo]{Corollary}

\newtheorem{rema}[theo]{Remark}

\newtheorem{exam}[theo]{Example}

\begin{document}
\title[A Noether-Lefschetz theorem]
{A Noether-Lefschetz theorem for varieties of $r$-planes
in complete intersections}
\date{\today}
\author[Z. Jiang]{Zhi Jiang}
\address{Max-Planck-Institut f\"{u}r Mathematik\\Vivatsgasse 7\\53111, Bonn}
\email{{\texttt flipz@mpim-bonn.mpg.de}}

\maketitle
Let $X$ be a general complete intersection in complex projective space. The Picard number of $X$ is known. We may state it in the following form.

\medskip
\noindent{\bf Theorem 1} {\em Let $X$ be a smooth complete intersection of dimension $\geq 2$ in complex projective space.
\begin{itemize}
\item if $\dim X\geq 3$, the second Betti number of $X$ is $1$, and in particular, the Picard number $\rho(X)$ of $X$ is 1;
\item if $\dim X=2$ and $X$ is very general, the Picard number $\rho(X)$ is $1$, except when $X$ is a quadric surface in $\mathbf{P}^3$, or $X$ is a cubic surface in $\mathbf{P}^3$, or $X$ is a complete intersection of two quadrics in $\mathbf{P}^4$.
\end{itemize}}
\medskip

The first part of the above theorem comes from the Lefschetz hyperplane theorem and the second part is the so-called Noether-Lefschetz theorem (see \cite[section 3]{V}, or \cite{Sp}, section 2). In \cite{BV}, Bonavero and Voisin, using Deligne's global invariant cycles theorem, proved an analogue of the first part of the above theorem for the schemes parametrizing $r$-planes contained in a complete intersection in complex projective space. Debarre and Manivel later used Bott's theorem to give another proof of the same theorem in \cite{DM}. We will first recall their theorem.

We follow the presentation in \cite{DM}. Let $V$ be a complex vector space of dimension $n+1$. For a finite sequence $\underline{d}=(d_1,\ldots, d_s)$ of integers $\geq 2$ and a positive integers $r$, we set $|\underline{d}|=\sum_{i=1}^sd_i$, $\underline{d}+r=(d_1+r,\ldots, d_s+r)$, and $\binom{\underline{d}}{r}=\sum_{i=1}^s\binom{d_i}{r}$. We then set $$\delta(n,\underline{d}, r)=(r+1)(n-r)-\binom{\underline{d}+r}{r},$$ and $\delta_{-}(n,\underline{d},r)=\min\{\delta(n,\underline{d},r), n-2r-s\}$.

Let $X\subset \mathbf{P}(V)$ be a complete intersection defined by $f_{1}=\cdots=f_s=0$, where $f_i\in \Sym^{d_i}V^*$ for each $1\leq i\leq s$.  We denote by $F_r(X)$ the subscheme of $G:=G(r+1, V)$ parametrizing the linear spaces of dimension $r$ contained in $X$. On $G(r+1, V)$, there is the tautological sub-bundle $\Sigma$ of $V\otimes\cO_G$ of rank $r+1$ and the tautological quotient bundle $Q$ of rank $n-r$, and $\Omega_{G}^1\simeq \Sigma\otimes Q^*$.  Each $f_i$ induces a global section $\sigma_i$ of $\Sym^{d_i}\Sigma^*$ on $G$. We can also see $F_r(X)$ as the zero locus of the global sections $\sigma_i$ for all $1\leq i\leq s$.

Debarre and Manivel in \cite{DM} (see also Proposition 4 and Proposition 5 in \cite{BV}) proved the following theorem.

\medskip
\noindent{\bf Theorem 2}\label{thmDM} {\em Assume $X$ is a general complete intersection as above. If $\delta_{-}(n,\underline{d},r)\geq 1$, the scheme $F_{r}(X)$ is connected, smooth and of dimension $\delta(n,\underline{d}, r)$. Furthermore, if $\delta_{-}\geq 3$, the second Betti number of $F_{r}(X)$ is $1$ and in particular, the Picard number $\rho(F_r(X))$ of $F_r(X)$ is 1.}
\medskip

 The above theorem is optimal. Indeed, when $\delta_{-}(n,\underline{d},r)\leq 2$, the Hodge number $h^{2,0}(F_r(X))$ is often non-zero. Hence though we know the second Betti number of $F_r(X)$,  it is a priori not clear what the Picard number is.

 The main theorem of this paper is the following:

\medskip
\noindent{\bf Theorem 3}\label{thm1} {\em Let $X$ be a very general complete intersection in projective space. Assume that $\delta(n, \underline{d}, r)\geq 2$ and $\delta_{-}(n, \underline{d}, r)>0$.\footnote{This assumption is not important. We only use it to exclude the case when $X$ is a quadric in $\mathbf{P}^{2r+1}$ for $r\geq 1$, where $F_r(X)$ has two smooth isomorphic disjoint components and the Picard number of each component is $1$.} We have $\rho(F_r(X))=1$ except in the following cases:
\begin{itemize}
\item $X$ is a quadric in $\mathbf{P}^{2r+3}$, $r\geq 1$, where  the Picard number of $F_r(X)$ is $2$;
\item $X$ is a complete intersection of two quadrics in $\mathbf{P}^{2r+4}$, $r\geq 1$, where the Picard number of $F_r(X)$ is $2r+6$.
\end{itemize}. }

In the following examples, we will see that when $\delta_{-}(n, \underline{d}, r)\leq 2$, the Picard number $\rho(F_r(X))$ varies, where the situation is similar to the families of smooth hypersurfaces of degree $\geq 4$ in $\mathbf{P}^3$.

\begin{exam}\label{exam1}\upshape
 Let $X\hookrightarrow \mathbf{P}^{2r+3}$ be a general complete intersection of two quadrics, where $r\geq 1$. Considering
 $F_r(X)$ and in this case we have $\delta(2r+3,\underline{d},r)=r+1$ and $\delta_{-}(2r+3,\underline{d},r)=1$.

We may assume that the two quadrics defining $X$ are \begin{equation*}\sum_{i}x_i^2\;\; \textrm{and}\;\sum_i\lambda_ix_i^2\end{equation*} where $\lambda_i\neq \lambda_j$ if $i\neq j$. By \cite{Re}, we know that $F_r(X)$ is an abelian variety and is isomorphic to the Jacobian of the hyperelliptic curve defined by $$y^2=\Pi_i(x_0-\lambda_ix_1).$$ Therefore by \cite{P}, $\rho(F_r(X))=1$ for a very general $X$.
There also exists smooth $X$ such that $\rho(F_r(X))\geq 2$, for instance, when the hyperelliptic curve is defined by $y^2=x^{2r+4}-1$.
\end{exam}

\begin{exam}\label{exam2}\upshape [See \cite{CG}] Let $X\hookrightarrow \mathbf{P}^4$ be a cubic threefold. We consider $F_1(X)$ and have in this case $\delta(4, 3, 1)=2$ and $\delta_{-}(4, 3, 1)=1$.

Clemens and Griffiths proved that the Abel-Jacobi map $\alpha: F_1(X)\rightarrow JX$ induces an isomorphism $\alpha_*: \Alb(F_1(X))\simeq JX$ and $\wedge^2 H^1(F_1(X), \mathbf{Q})\simeq H^2(F_1(X), \mathbf{Q}).$ We conclude that $$H^2(F_1(X), \mathbf{Q})\simeq H^2(JX, \mathbf{Q})\simeq \wedge^2H^3(X, \mathbf{Q}).$$

Since the Zariski closure of the monodromy group for the family of cubic threefolds is the whole symplectic group (\cite{PS}, Theorem 10.22), we have $\rho(F_1(X))=1$ for $X$ very general. We see in \cite{Ro} that there is a $7$-dimensional family in the $10$-dimensional moduli space of cubic threefolds parametrizing $X$ whose associated $F_1(X)$ contains elliptic curves. For such $X$, we have $\rho(F_1(X))\geq 2$.
\end{exam}

\begin{exam}\label{exam3}\upshape Let $X\hookrightarrow \mathbf{P}^7$ be a cubic $6$-fold. We consider $F_2(X)$ and have $\delta(7, 3, 2)=5$ and $\delta_{-}(7, 3, 2)=2$.

The $2$-planes cover the whole of $X$, hence the Abel-Jacobi map induces an injective map
$H^6(X, \mathbf{Q})_{\textrm{prim}}\rightarrow H^2(F_2(X), \mathbf{Q})$. We can also compute that $$\dim H^6(X, \mathbf{Q}) =\dim H^2(F_2(X), \mathbf{Q})=87.$$ Hence since the monodromy group of the family of cubic of dimension $6$ is again big (\cite{PS}, Theorem 10.22), if $X$ is general, $\rho(F_2(X))=\textrm{rank}(H^{3,3}(X)\cap H^6(X, \mathbf{Q}))=1$. If $X$ contains some special codimension-$3$ subvariety, e.g. $\mathbf{P}^3$, we have $\rho(F_2(X))=\textrm{rank}(H^{3,3}(X)\cap H^6(X, \mathbf{Q}))\geq 2$.
\end{exam}

We will give an application of our Noether-Lefschetz type theorem.

For any smooth cubic fivefold $Z$, we denote by $(JZ, \Theta)$ the principally polarized intermediate Jacobian of $Z$ which is a $21$-dimensional principally polarized abelian variety. We then have the Abel-Jacobi map $$\alpha: F_2(Z)\rightarrow JZ.$$ As an application of our main theorem, we have the following:

\medskip
\noindent{\bf Theorem 4}  {\em  With the notations as above, the cohomology class $$[\alpha_*(F_2(Z))]=12[\frac{\Theta^{19}}{19!}].$$
}

\medskip

{\small\noindent{\bf Acknowledgements.}
We are indebted to Olivier Debarre, Laurent Manivel, and Claire Voisin for helpful discussions. We also thank Fr\'{e}d\'{e}ric Han for help with the Program Lie.}

\medskip
\section{Preliminaries}\label{section1}
We shall use variations of Hodge structures to prove the main theorem. In this section, we will first recall a lemma which reduces the proof to the surjectivity of some maps between cohomology groups and then we will recall Bott's theorem which helps us to calculate the cohomology groups of homogeneous vector bundles on Grassmannians.

Let $Y$ be a smooth projective variety of dimension $N$ and Picard number $1$. Let $W$ be a vector bundle of rank $R$ on $Y$ which is globally generated with $N-R\geq 2$. For a section $\sigma$ of $W$, we denote by $X_{\sigma}$ the zero locus of $s$.
\begin{lemm}\label{hodge}We fix a general section $\sigma$ of $W$. Assume that $X_{\sigma}$ is smooth and of the expected dimension and assume that we have
\begin{itemize}
\item[a)] $\dim H^1(X_{\sigma}, \Omega_{Y}^1|_{X_{\sigma}})=1$;
\item[b)] $H^{N-R-2}(X_{\sigma}, K_{X_{\sigma}})\otimes H^0(Y, W)|_{X_{\sigma}}\rightarrow H^{N-R-2}(X_{\sigma}, W\otimes K_{X_{\sigma}})$ is surjective.
\end{itemize}
Then for a very general section $\rho\in H^0(Y, W)$, the Picard number of the zero locus $X_{\rho}$ is $1$.
\end{lemm}
\begin{proof}Let $U\subset H^0(Y, W)$ be the open subset parametrizing sections whose zero locus is smooth of the expected dimension. We set $\pi: \mathcal{X}\rightarrow U$ be the family of $X_{\sigma}$. We have the sequence:
\begin{eqnarray*}\label{sequence1}0\rightarrow T_{X_{\sigma}}\rightarrow T_{\mathcal{X}}|_{X_{\sigma}}\rightarrow T_{U,\sigma}\rightarrow 0.
\end{eqnarray*}
And for any $t\in T_{U,\sigma}$, we have the Kodaira-Spencer class $\delta(t)\in H^1(X_{\sigma}, T_{X_{\sigma}})$, where $\delta$ is the coboundary map in (\ref{sequence1}).

We denote by $h$ the cohomology class of an ample divisor of $Y$, and denote by $H^{1,1}(X_{\sigma})_{\textrm{prim}}$ the primitive $(1,1)$-forms on $X_{\sigma}$ relative to $h$. In order to prove the lemma, we just need to show that for any $\lambda\in H^{1,1}(X_{\sigma})_{\textrm{prim}}$, the map \begin{eqnarray}\label{map1}\nabla_{\lambda}\circ\delta: H^0(Y, W)\simeq T_{U,\sigma}\rightarrow H^2(X_{\sigma}, \cO_{X_{\sigma}})\end{eqnarray} is non-trivial (see for instance \cite[Chapter 5]{V}).

Considering the sequence $$0\rightarrow W^*\rightarrow \Omega^1_{Y}|_{X_{\sigma}}\rightarrow \Omega^1_{X_{\sigma}}\rightarrow 0,$$ we have $$H^1(X_{\sigma}, \Omega^1_{Y}|_{X_{\sigma}})\rightarrow H^1(X_{\sigma}, \Omega_{X_{\sigma}}^1)\xrightarrow{\delta} H^2(X_{\sigma}, W^*).$$ Since $\dim H^1(X_{\sigma}, \Omega^1_{Y}|_{X_{\sigma}})=1$, we have that the map $$\delta: H^{1,1}(X_{\sigma})_{\textrm{prim}}\rightarrow H^2(X_{\sigma}, W^*)$$ is injective. Thus by Serre duality of the map (\ref{map1}), we just need to show that the map $$H^{N-R-2}(X_{\sigma}, K_{X_{\sigma}})\otimes H^0(Y, W)|_{X_{\sigma}}\rightarrow H^{N-R-2}(X_{\sigma}, W\otimes K_{X_{\sigma}})$$ is surjective to conclude the Picard number of a general zero locus $X_{\rho}$ is $1$.
\end{proof}
We then recall Bott's theorem. Since we only work on $G(r+1, V)$ in this paper, we will present Bott's theorem in an elementary way (see for instance \cite[section 4.6]{Sp}).

We call a finite decreasing sequence of integers $c=(c_1,\ldots, c_k)$ a partition. For a vector space $W$ of dimension $k$, we denote by $\Gamma^cW$ the irreducible $\textrm{GL}(W)$-module. For example, we have \begin{eqnarray*}\Gamma^{(k,0,\ldots,0)}W=\sym^kW, \end{eqnarray*}
\begin{eqnarray*}\Gamma^{(0,0,\ldots,-k)}W=\sym^kW^*, \end{eqnarray*}
\begin{eqnarray*}\Gamma^{(1,1,\ldots,1)}W=\det W. \end{eqnarray*}

Let $b=(b_1, b_2, \ldots, b_{n-r})$ and $a=(a_1, a_2, \ldots, a_{r+1})$ be partitions.  We consider \begin{eqnarray*}\phi(b,a)&=&(b_1,b_2, \ldots, b_{n-r}, a_1, a_2, \ldots, a_{r+1})-(1,2,\ldots, n+1)\\
&=&(c_1,c_2, \ldots, c_{n+1}).\end{eqnarray*} We define the set $\Phi_{b,a}$ to be the set consisting of the pairs  $(s, t)$ with $s<t$ and $c_s>c_t$. We will denote by $i(b, a)$ the cardinality of $\Phi_{b, a}$. We then re-order $\phi(b,a)$ by making it decreasing and denote the resulting sequence by $\phi(b,a)^+$ and set $$\psi(b, a):=\phi(b,a) ^++(1,2,\ldots, n+1).$$ If $\psi(b, a)$ is not decreasing, we have $\Gamma^{\psi(b, a)}V=0$ by convention. We have Bott's theorem:
\begin{theo}\label{Bott} With the above notations, we have \begin{itemize}
\item[1)]$H^q(G, \Gamma^bQ\otimes \Gamma^a\Sigma)=0$ if $q\neq i(b, a)$;
\item[2)]$H^{i(b, a)}(G, \Gamma^bQ\otimes \Gamma^a\Sigma)=\Gamma^{\psi(b, a)}V$.
\end{itemize}
\end{theo}
Here is a useful corollary.
\begin{coro}On the Grassmannian $G=G(r+1, V)$, assume that for some partition $a$, we have $$H^q(G, \Gamma^a\Sigma)\neq 0.$$ Then there exists $k\geq 0$ such that $q=k(n-r)$.
\end{coro}

We will often identify the partitions $(a_1,\ldots,a_{n-r})$ and $(a_1,\ldots,a_{n-r},0)$. The following lemma is crucial.
\begin{lemm}\label{lemma1}We keep the notations as above,
\begin{itemize}
\item[1)] let $a_1$ and $a_2$ be partitions $\geq 0$, the multiplication map $$H^0(G, \Gamma^{a_1}\Sigma^*)\otimes H^0(G, \Gamma^{a_2}\Sigma^*)\rightarrow H^0(G, \Gamma^{a_1}\Sigma^*\otimes\Gamma^{a_2}\Sigma^*)$$ is surjective;
\item[2)]
for some positive partition $a>0$ and some positive integer $d>0$,
if there exists some $k\geq 1$ such that the multiplication $$H^{k(n-r)}(G, \Gamma^{a}\Sigma)\otimes H^0(G, \Sym^d\Sigma^*)\rightarrow H^{k(n-r)}(G, \Gamma^{a}\Sigma\otimes\Sym^{d}\Sigma^*)$$ is not surjective, we have $a_k\geq n-r+k$ and $a_{k+1}\geq k+1$.
\end{itemize}
\end{lemm}
\begin{proof}By Theorem \ref{Bott}, we have $H^0(G, \Gamma^{a_1}\Sigma^*)=\Gamma^{a_1}V^*$ and $H^0(G, \Gamma^{a_2}\Sigma^*)=\Gamma^{a_2}V^*$. We have by the Littlewood-Richardson rule that
\begin{eqnarray*}\Gamma^{a_1}\Sigma^*\otimes\Gamma^{a_2}\Sigma^*=\bigoplus_vN_{v}\cdot\Gamma^v\Sigma^*,\\
\Gamma^{a_1}V^*\otimes\Gamma^{a_2}V^*=\bigoplus_vN_{v}'\cdot\Gamma^vV^*.
\end{eqnarray*}
Since $\dim V^*>\rank\Sigma^*$, $N_{v}'\geq N_{v}$ for any partition $v$. Then we conclude the proof of $1)$ by Theorem \ref{Bott}.

We write \begin{eqnarray*}\Gamma^a\Sigma\otimes \Sym^d\Sigma^*=\bigoplus_{\alpha}\Gamma^{\alpha}\Sigma,\end{eqnarray*} where by the Littlewood-Richardson rule, $\alpha$ goes through all partitions satisfying $$a_1\geq \alpha_1\geq a_2\geq \alpha_2\geq\cdots\geq a_{r+1}\geq \alpha_{r+1},$$ and $$\sum_{i=1}^{r+1}\alpha_i=\sum_{i=1}^{r+1}a_i-d.$$

Fix any such partition $\alpha$.

If $a_k<n-r+k$, we have $\alpha_k<n-r+k$. Hence by Bott's theorem, we have $$H^{k(n-r)}(G, \Gamma^{a}\Sigma)=H^{k(n-r)}(G, \Gamma^{a}\Sigma\otimes\Sym^{d}\Sigma^*)=0,$$ which contradicts our hypothesis.

If $a_{k+1}\leq k$, $H^{k(n-r)}(G, \Gamma^a\Sigma)\simeq \Gamma^{\psi(0, a)}V\neq 0$. We now study the map
\begin{multline*}m_{\alpha}: H^{k(n-r)}(G, \Gamma^a\Sigma)\otimes H^0(G, \Sym^d\Sigma^*)\rightarrow H^{k(n-r)}(G, \Gamma^a\Sigma\otimes\Sym^d\Sigma^*)\\\rightarrow H^{k(n-r)}(G, \Gamma^{\alpha}\Sigma)\end{multline*}

Now we denote by $V_1$ a linear subspace of codimension-($r+1-k)$ of $V$ and take a linear subspace $V_2$ of $V$ so that $V\simeq V_1\oplus V_2$. Let $G_1:=G(k, V_1)$ be the Grassmannian and denote by $i: G_1\rightarrow G$ the natural embedding which sends a subspace $\Sigma_1$ of $V_1$ to a subspace $\Sigma_1\oplus V_2$ of $V$. Then we denote by $\overline{a}$ (resp. $\overline{\alpha}$) the partition $(a_1,\ldots,a_k)$ (resp. $(\alpha_{1},\ldots,\alpha_{k})$) and denote by $\widetilde{a}$ (resp. $\widetilde{\alpha}$) the partition $(a_{k+1},\ldots, a_{r+1})$ (resp. $(\alpha_{k+1},\ldots, \alpha_{r+1})$). Hence $a=(\overline{a}, \widetilde{a})$ and $\alpha=(\overline{\alpha}, \widetilde{\alpha})$. Finally, we denote $d_1=\sum_{i=1}^{k}(a_i-\alpha_i)$ and $d_2=\sum_{i=k+1}^{r+1}(a_i-\alpha_i)$. We have both $d_1$ and $d_2$ are $\geq 0$ and $d_1+d_2=d$.
We notice that $\Gamma^{\overline{a}}\Sigma_1\otimes\Gamma^{\widetilde{a}}V_2$ (resp. $\Sym^{d_1}\Sigma_1^*\otimes\Sym^{d_2}V_2^*$) is a direct summand of $i^*\Gamma^{a}\Sigma$ (resp. $i^*\Sym^d\Sigma^*$).

We notice that $i(G_1)$ is the zero locus of a global section of $Q\otimes V_2^*$ in $G$. Hence,
by Bott's theorem, the restriction maps \begin{eqnarray}\label{restriction}\nonumber&r_1&: H^{k(n-r)}(G, \Gamma^a\Sigma)\simeq \Gamma^{\psi(0,a)}V\rightarrow H^{k(n-r)}(G_1, i^*\Gamma^a\Sigma)\\ &r_2&: H^0(G, \Sym^d\Sigma^*)\rightarrow H^0(G_1, i^*\Sym^d\Sigma_1^*)\\\nonumber &r_3&: H^0(G, \Gamma^{\alpha}\Sigma)\rightarrow H^0(G_1, i^*\Gamma^{\alpha}\Sigma)\end{eqnarray} are surjective.

We now consider on $G_1$ the multiplication \begin{multline*} m_{\overline{\alpha}, \widetilde{\alpha}}: H^{k(n-r)}(G_1, \Gamma^{\overline{a}}\Sigma_1\otimes\Gamma^{\widetilde{a}}V_2)\otimes H^0(G_1, \Sym^{d_1}\Sigma_1^*\otimes\Sym^{d_2}V_2^*)\\\rightarrow H^{k(n-r)}(G_1, \Gamma^{\overline{\alpha}}\Sigma_1\otimes\Gamma^{\widetilde{\alpha}}V_2),\end{multline*} where both sides are non-zero.
Take the Serre duality, the multiplication \begin{multline*} H^{0}(G_1, \Gamma^{\overline{\alpha}}\Sigma_1^*\otimes K_{G_1}\otimes\Gamma^{\widetilde{\alpha}}V_2^* )\otimes H^0(G_1, \Sym^{d_1}\Sigma_1^*\otimes\Sym^{d_2}V_2^*)\\\rightarrow H^{k(n-r)}(G_1, \Gamma^{\overline{a}}\Sigma_1^*\otimes K_{G_1}\otimes\Gamma^{\widetilde{a}}V_2^*)\end{multline*} is surjective by the statement $1)$ of this lemma.
Therefore, the map $m_{\overline{\alpha}, \widetilde{\alpha}}$ is non-trivial.

We then consider the commutative diagram:
\begin{eqnarray*}
\xymatrix{
H^{k(n-r)}(G, \Gamma^a\Sigma)\otimes H^0(G, \Sym^d\Sigma^*)\ar[d]^{r_1\otimes r_2}\ar[r]^{m_{\alpha}}& H^{k(n-r)}(G, \Gamma^{\alpha}\Sigma)\ar[d]^{r_3}\\
H^{k(n-r)}(G_1, i^*\Gamma^a\Sigma)\otimes H^0(G_1, i^*\Sym^d\Sigma^*)\ar[d]_{\textrm{projection to direct summand}}\ar[r]^{m_{\alpha}|_{G_1}}& H^{k(n-r)}(G_1, i^*\Gamma^{\alpha}\Sigma))\ar[d]_{\textrm{projection to direct summand}}\\
H^{k(n-r)}(G_1, \Gamma^{\overline{a}}\Sigma_1\otimes\Gamma^{\widetilde{a}}V_2)\otimes H^0(G_1, \Sym^{d_1}\Sigma_1^*\otimes\Sym^{d_2}V_2^*)\ar[r]^(.6){m_{\overline{\alpha}, \widetilde{\alpha}} }& H^{k(n-r)}(G_1, \Gamma^{\overline{\alpha}}\Sigma_1\otimes\Gamma^{\widetilde{\alpha}}V_2)
}
\end{eqnarray*}

Since we know by $(\ref{restriction})$ that all the vertical maps are surjective and we have seen above that the map $m_{\overline{\alpha},\widetilde{\alpha}}$ is non-trivial, we deduce that $m_{\alpha}$ is also non-trivial. Moreover, since $\Gamma^{\psi(0,\alpha)}V$ is an irreducible $\textrm{GL}(V)$-module, $m_{\alpha}$ is surjective.

Since the multiplicity of each $\Gamma^{\alpha}\Sigma$ in $\Gamma^a\Sigma\otimes \Sym^d\Sigma^*$ is $1$ and $\Gamma^{\psi(0, \alpha)}V$ is an irreducible $\textrm{GL}(V)$-module, by Schur's Lemma, we deduce that the multiplication $$H^{k(n-r)}(G, \Gamma^{a}\Sigma)\otimes H^0(G, \Sym^d\Sigma^*)\rightarrow H^{k(n-r)}(G, \Gamma^{a}\Sigma\otimes\Sym^{d}\Sigma^*)$$ is again surjective.

This concludes the proof.
\end{proof}
\begin{rema}\upshape
Assume we have $H^i(G, \Gamma^{\alpha}\Sigma)\neq 0$ and $H^j(G, \Gamma^{\beta}\Sigma)\neq 0$, for partitions $\alpha, \beta$ and integers $i,j>0$. The multiplication $$H^i(G, \Gamma^{\alpha}\Sigma)\otimes H^j(G, \Gamma^{\beta}\Sigma)\rightarrow H^{i+j}(G, \Gamma^{\alpha}\Sigma\otimes \Gamma^{\beta}\Sigma)$$ is in general NOT surjective.

We may consider in $G:=G(2,4)$. Let $\alpha=\beta=(3,1)$. Then $\Gamma^{(4,4)}\Sigma$ is a direct summand of $\Gamma^{(3,1)}\Sigma\otimes\Gamma^{(3,1)}\Sigma$. However, through restrictions on $G(2,3)\hookrightarrow G$, we can see that the multiplication $$H^2(G, \Gamma^{(3,1)}\Sigma)\otimes H^2(G, \Gamma^{(3,1)}\Sigma)\rightarrow H^4(G, \Gamma^{(4,4)}\Sigma)$$ is $0$.

The cup products of line bundles on homogeneous varieties have been studied intensively in \cite{DT}.
\end{rema}
\section{Proof of the main theorem when $\dim(F_r(X))=2$}
Under the assumptions of the main theorem, we will assume furthermore that $\delta(n, \underline{d}, r)=2$ in this section. We notice that $F_r(X)$ is smooth and connected.

\begin{prop}\label{main prop}Assume $\delta(n, \underline{d}, r)=2$. Set $G:=G(r+1, V)$. Then, for any $1\leq i\leq s$, the map $$H^0(F_r(X), K_{F_r(X)})\otimes H^0(G, \Sym^{d_i}\Sigma^*)\rightarrow H^0(F_r(X), \Sym^{d_i}\Sigma^*\otimes K_{F_r(X)})$$ is surjective, except in the following cases:
\begin{itemize}
\item $X$ is a complete intersection of two quadrics in $\mathbf{P}^5$ and $r=1$;
\item $X$ is a cubic in $\mathbf{P}^4$ and $r=1$.
\end{itemize}
\end{prop}
\begin{proof}Denote by $H$ the Pl\"{u}cker polarization on $G$. Set $$M=\binom{\underline{d}+r}{r+1}-n-1.$$ Then $K_{F_r(X)}=MH$. Since $\delta(n, \underline{d}, r)=2$, we can verify that $M\geq 0$ with equality only when $n=5$, $r=1$, and $\underline{d}=(2,2)$.

Let $W$ be the vector bundle $$\bigoplus_{i=1}^s\Sym^{d_i}\Sigma$$ on $G$. We have the following resolution for the structure sheaf $\cO_{F_r(X)}$:
$$0\rightarrow \wedge^{\binom{\underline{d}+r}{r}}W\rightarrow\cdots\rightarrow\wedge^{n-1}W\rightarrow\cdots\rightarrow W\rightarrow\cO_{G}\rightarrow\cO_{F_r(X)}\rightarrow 0$$
It follows that there is a decreasing filtration $F_{*}$ on $H^0(F, \Sym^{d_i}\Sigma^*\otimes K_{F_r(X)})$ such that $$H^{k(n
-r)}(G, \wedge^{k(n-r)}W\otimes\Sym^{d_i}\Sigma^*\otimes\cO_G(MH))\twoheadrightarrow F_k/F_{k+1},$$
and in order to prove the proposition, we just need to prove that the multiplication
\begin{multline*}\label{multication}
\Psi_{k,i}: H^{k(n-r)}(G, \wedge^{k(n-r)}W\otimes\cO_G(MH))\otimes H^0(G, \Sym^{d_i}\Sigma^*)\rightarrow\\H^{k(n-r)}(G, \Sym^{d_i}\Sigma^*\otimes \wedge^{k(n-r)}W\otimes\cO_G(MH))\end{multline*} is surjective for all $0\leq k\leq r$ and $1\leq i\leq s$.

We now assume that there exist $k$ and $i$ such that $\Psi_{k,i}$ is not surjective.

Considering the decomposition of the vector bundle $\wedge^{k(n-r)}W$ on $G$, we may write  $$\wedge^{k(n-r)}W=\bigoplus_{\alpha}\Gamma^{\alpha}\Sigma.$$ Then according to Lemma \ref{lemma1}, $k\geq 1$ and there exists $\alpha$ such that \begin{eqnarray}\label{inequality1}&&\alpha_{k}\geq M+n-r+k,\\\label{inequality2}&& \alpha_{k+1}\geq M+k+1.\end{eqnarray} Moreover, we notice that $M+n-r+k=\binom{\underline{d}+r}{r+1}-r-1+k$ and $$\det W= -\binom{\underline{d}+r}{r+1}H= \Gamma^{\big(\binom{\underline{d}+r}{r+1}, \ldots, \binom{\underline{d}+r}{r+1}\big)}\Sigma.$$ Hence by Lemma \ref{representation}  (take $s=r+1-k$), we have the following inequality:
\begin{eqnarray*}\label{ingredient}k(n-r)\geq \left[\binom{r+\underline{d}}{r}-\binom{r-k+\underline{d}}{r-k}\right]-k(r+1-k).\end{eqnarray*}
Since $\delta(n, \underline{d}, r)=2$,
we conclude that  \begin{eqnarray}\label{inequality3}&&\frac{1}{r+1} \binom{r+\underline{d}}{r}+\frac{2}{r+1}+r+1-k\\\nonumber\geq &&\frac{1}{k} \left[\binom{r+\underline{d}}{r}-\binom{r-k+\underline{d}}{r-k}\right].\end{eqnarray}

Since $k\leq r$, the function $$\frac{1}{k} \left[\binom{r+\underline{d}}{r}-\binom{r-k+\underline{d}}{r-k}\right]-\frac{1}{r+1} \binom{r+\underline{d}}{r}>0$$ is a strictly increasing function of each $d_i$.
%
%
%
by induction on $d_i$, we can check that the inequality (\ref{inequality3}) holds for some $k$ only in the following cases:
\begin{itemize}\label{diagram2}
\item  $r=1$
\item[1.1)]$\underline{d}=(2,2)$, $M=0$, and $n=5$;
\item[1.2)]$\underline{d}=(2,2,2,2)$, $M=3$, and $n=8$;
\item[1.3)] $\underline{d}=3$, $M=1$, and $n=4$;
\item[1.4)] $\underline{d}=(3,3)$, $M=5$, and $n=6$;
\item[1.5)] $\underline{d}=(3,2,2)$, $M=4$, and $n=7$;
\item[1.6)] $\underline{d}=(4,2)$, $M=6$ and $n=6$;
\item[1.7)] $\underline{d}=5$, $M=9$, and $n=5$;
\item $r=2$
\item[2.1)] $\underline{d}=3$, $M=3$, and $n=6$;
\item[2.2)] $\underline{d}=(3,2)$, $M=5$, and $n=8$;
\item $r=3$
\item[3)]$\underline{d}=(2,2,2)$, $M=3$, and $n=11$.

\end{itemize}

In the next step, we shall check in the above cases whether the inequalities $(\ref{inequality1})$  and $(\ref{inequality2})$ actually hold for some $\alpha$. We will use the Program Lie to get precise information of the decompositions of $\wedge^{k(n-r)}W$ and then show that these inequalities do not hold except in cases $1.1)$ and $1.3)$. This will conclude the proof of the proposition.
Since this is only a computation, we just show the case $r=3$.

In case 3), we have $n=11$, $r=3$, and $M=3$.

When $k=1$, we have $\alpha_1\geq 12$ and $\alpha_2\geq 5$ by (\ref{inequality1}) and (\ref{inequality2}). However the total weight for $\wedge^8W=\wedge^8(\Sym^2\Sigma\oplus\Sym^2\Sigma\oplus\Sym^2\Sigma)$ is only $16$: this is impossible.

When $k=3$, we have $\alpha_1\geq \alpha_2\geq \alpha_3\geq 14$, and $\alpha_4\geq 7.$ But the total weight of $\wedge^{24}(\Sym^2\Sigma\oplus\Sym^2\Sigma\oplus\Sym^2\Sigma)$ is $48$. This is again impossible.

We have the decompositions $$\wedge^t\Sym^2\Sigma=\oplus_{i}\Gamma^{\alpha^t_i}\Sigma,$$
where
\begin{itemize}
\item $t=1$, $\alpha_1^1=(2,0,0,0)$;
\item $t=2$, $\alpha_1^2=(3,1,0,0)$;
\item $t=3$, $\alpha_1^3=(4,1,1,0)$, and $\alpha_2^3=(3,3,0,0)$;
\item $t=4$, $\alpha_1^4=(4,3,1,0)$, and $\alpha_2^4=(5,1,1,1)$;
\item $t=5$, $\alpha_1^5=(5,3,1,1)$, and $\alpha_2^5=(4,4,2,0)$;
\item $t=6$, $\alpha_1^6=(5,4,2,1)$, and $\alpha_2^6=(4,4,4,0)$;
\item $t=7$, $\alpha_1^7=(5,4,4,1)$, and $\alpha_2^7=(5,5,2,2)$;
\item $t=8$, $\alpha_1^8=(5,5,4,2)$;
\item $t=9$, $\alpha_1^9=(5,5,5,3)$;
\item $t=10$, $\alpha_1^{10}=(5,5,5,5)$.
\end{itemize}

When $k=2$, we have $\alpha_1\geq \alpha_2\geq 13$ and $\alpha_3\geq 6$. Hence $\Gamma^{(13,13,6,0)}\Sigma$ should be a subbundle of $\wedge^{16}W=\wedge^{16}(\Sym^2\Sigma\oplus\Sym^2\Sigma\oplus\Sym^2\Sigma)$. We have $$\wedge^{16}(\Sym^2\Sigma\oplus\Sym^2\Sigma\oplus\Sym^2\Sigma)=
\bigoplus_{\substack{(s_1,s_2,s_3)\\s_1+s_2+s_3=16}}\wedge^{s_1}\Sym^2\Sigma\otimes \wedge^{s_1}\Sym^2\Sigma\otimes\wedge^{s_3}\Sym^2\Sigma.$$
By the above list, we see that any $\alpha_i^t$ with $(\alpha_i^t)_4=0$ has $(\alpha_i^t)_1\leq 4$.
Hence we exclude this case by the Littlewood-Richardson rule.

We have thus excluded the case $3)$.

\end{proof}

\begin{lemm}\label{representation}Let $V$ be a vector space of dimension $r+1$ and let $\underline{d}=(d_1,\ldots,d_s)$ be a sequence of integers $\geq 2$. Denote by $W$ the vector space $ \Sym^{d_1}V\oplus\cdots\oplus\Sym^{d_s}V$.
For any integer $s\geq 0$ and $1\leq k\leq r$, we take an integer $t$ such that $$0<t<\left[\binom{\underline{d}+r}{r}-\binom{\underline{d}+r-k}{r-k}\right]-ks.$$ Then for any irreducible component $\Gamma^{\lambda}V$ of $\wedge^tW$, we have $$\lambda_k< \binom{\underline{d}+r}{r+1}-s.$$
\end{lemm}
The proof of this lemma is parallel to the proof of Lemme 3.9 in \cite{DM}. For reader's convenience, we give the details.
\begin{proof}For simplicity, we assume $\underline{d}=d$.
We denote by $X$ the Grassmannian parametrizing subspaces of  dimension $r+1-k$ of $V$ and denote by $Y$
the Grassmannian parametrizing subspaces of dimension $\binom{d+r-k}{r-k}$ of $W$. We denote respectively  by $\Sigma_X$ (resp. $Q_X$) and $\Sigma_Y$ (resp. $Q_Y$) the tautological subbundle (resp. quotient bundle) on $X$ and $Y$. There is a natural embedding $i: X\hookrightarrow Y$ so that $i^*\Sigma_Y=\Sym^d\Sigma_X$. Let $N$ be the normal bundle of $X\hookrightarrow Y$.

By Bott's theorem, we have $H^0(Y, \wedge^tQ_Y)=\wedge^tW$. Considering the exact sequences:
$$0\rightarrow\cI_X\otimes \wedge^tQ_Y\rightarrow \wedge^tQ_Y\rightarrow\wedge^tQ_Y|_X\rightarrow 0,$$and for each $l\geq 1$,
$$0\rightarrow\cI_X^{l+1}\otimes \wedge^tQ_Y\rightarrow \cI_X^{l}\otimes\wedge^tQ_Y\rightarrow\wedge^tQ_Y|_X\otimes\Sym^lN^*\rightarrow 0,$$ we see that
there exists a filtration $(\Gamma_l)_{l\geq 0}$ on $\wedge^tW$ so that $$\Gamma_l/\Gamma_{l+1}\hookrightarrow H^0(X, i^*(\wedge^tQ_Y)\otimes\Sym^lN^*).$$

There is a filtration $(G_m)_{0\leq m\leq d-1}$ of $i^*Q_Y$ such that $$G_m/G_{m+1}=\Sym^{d-m}Q_X\otimes\Sym^m\Sigma_X.$$ Denote by $$U=\textrm{Gr}(i^*Q_Y)=\oplus_{m=0}^{d-1}\Sym^{d-m}Q_X\otimes\Sym^m\Sigma_X.$$
Hence $$\Gamma_l/\Gamma_{l+1}\hookrightarrow H^0(X, i^*(\wedge^tQ_Y)\otimes \Sym^lN^*)\hookrightarrow H^0(X, \wedge^tU\otimes \Sym^lN^*).$$

We now set $T=\binom{d+r}{r}-\binom{d+r-k}{r-k}$ the rank of $U$. Then $$\wedge^tU=\det U\otimes \wedge^{T-t}U^*.$$

By definition of $t$, we have $T-t>ks$, and hence considering the total weights, for any irreducible component $\Gamma^{\widehat{\alpha}}Q_X^*\otimes\Gamma^{\widehat{\beta}}\Sigma_X^*$ of $\wedge^{T-t}U^*$, we have $\widehat{\alpha}_1>s$. Moreover, we notice that $$\det U=(\det Q_X)^{\otimes \binom{d+r}{r+1}}\otimes (\det \Sigma_X)^{\otimes \big(\binom{d+r}{r+1}-\binom{d+r-k}{r-k+1}\big)}.$$ Therefore, for any irreducible component $\Gamma^{\alpha}Q_X\otimes\Gamma^{\beta}\Sigma_X$ of $\wedge^{T-t}U$, we have $$\alpha_k<\binom{d+r}{r+1}-s.$$
We notice that $N^*$ is a subbundle of $i^*\Omega_Y^1=\Sym^d\Sigma_X\otimes i^*Q_Y^*$. Hence, by Littlewood-Richardson rule, for any irreducible component $\Gamma^{\alpha}Q_X\otimes\Gamma^{\beta}\Sigma_X$ of $\wedge^{T-t}U\otimes \Sym^lN^*$, we still have $\alpha_k<\binom{d+r}{r+1}-s.$

Therefore, by Bott's theorem, for any irreducible component $\Gamma^{\lambda}V$ of $\wedge^tW$, we have $$\lambda_k< \binom{d+r}{r+1}-s.$$

 \end{proof}
\begin{lemm}\label{main lemm}If $\delta(n, \underline{d}, r)=2$, we have $$\dim H^1(F_r(X), \Omega_{G}^1|_{F_r(X)})=1,$$ except when $X\subset \mathbf{P}^5$ is a smooth complete intersection of two quadrics and $r=1$.
\end{lemm}
\begin{proof}We first notice that in the proof of Th\'{e}or\`{e}me 3.4 in \cite{DM}, the authors showed that $$\dim H^1(F_r(X), \Omega_{G}^1|_{F_r(X)})=1,$$ if $\delta_{-}(n, \underline{d}, r)\geq 2$. As we assume here $\delta(n, \underline{d}, r)=2$, the cases when $\delta_{-}(n, \underline{d}, r)= 1$ are
\begin{itemize}
\item $n=4$, $r=1$, $\underline{d}=3$;
\item $n=6$, $r=2$, $\underline{d}=3$;
\item $n=5$, $r=1$, $\underline{d}=(2,2)$.
\end{itemize}

We then need to show that in the first two cases $H^i(G, \Omega_G^1\otimes \cI_{F_r(X)})=0$ for $i=1,2$. Again denote by $W$ the vector bundle $\bigoplus_i\Sym^{d_i}\Sigma$. By the Koszul resolution of $\cI_{F_r(X)}$, we need to prove that $$H^{i+t-1}(G, \Omega_G^1\otimes \wedge^tW)=0,$$ for all $t\geq 1$ and $i=1,2$.
We can use the program Lie to check the decompositions of $\wedge^tW$ in each case. By Bott's theorem, we conclude the proof of the lemma.
\end{proof}

\begin{theo}\label{dim=2}Let $X$ be a very general complete intersection on $\mathbf{P}^n$ such that $\delta(n, \underline{d}, r)=2$. Then the Picard number $\rho(F_r(X))$ is $1$.
\end{theo}
\begin{proof}By Lemma \ref{hodge}, Proposition \ref{main prop}, and Lemma \ref{main lemm}, we have proved the theorem except when $r=1$ and $X$ is a complete intersection of two quadrics in $\mathbf{P}^5$ or $r=1$ and $X$ is a cubic threefold. These two remaining cases were studied in Example \ref{exam1} and Example \ref{exam2}.
\end{proof}
\section{Proof of the remaining cases}\label{section2}
In this section, we would like to check the remaining situations, namely when $\dim F_r(X)=\delta(n,\underline{d},r)\geq 3$ and $1\leq \delta_{-}(n,\underline{d},r)\leq 2$ (see Theorem 2). The list is the following:
\begin{itemize}\label{diagram3}
\item $d_1=2$:
\item[(Q.1)] $\underline{d}=(2)$, and $2r+2\leq n\leq 2r+3$;
\item[(Q.2)] $\underline{d}=(2, 2)$, and $2r+3\leq n\leq 2r+4$;
\item[(Q.3)] $\underline{d}=(2,2,2)$, $r=1$, $n=7$, and $\delta(n,\underline{d},r)=3$;
\item[(Q.4)] $\underline{d}=(2,2,2)$, $r=2$, $n=9$, and $\delta(n,\underline{d},r)=3$;
\item $d_1=3$:
\item[(C.1)] $\underline{d}=(3)$, $r=3$, $n=9$, and $\delta(n,\underline{d},r)=4$;
\item[(C.2)] $\underline{d}=(3)$, $r=2$, $n=7$, and $\delta(n,\underline{d},r)=5$;
\item[(C.3)] $\underline{d}=(3)$, $r=1$, $n=5$, and $\delta(n,\underline{d},r)=4$;
\item[(C.4)] $\underline{d}=(3,2)$, $r=1$, $n=6$, and $\delta(n,\underline{d},r)=3$;
\item $d_1=4$
\item[(Qr)] $\underline{d}=(4)$, $r=1$, $n=5$, and $\delta(n,\underline{d},r)=3$.
 \end{itemize}


We will first study case by case the complete intersections of less than $2$ quadrics and then discuss the others.
\subsection{Case (Q.1)}

In case $(\textrm{Q.1})$, $X\subset \mathbf{P}^{2r+3}$ is a smooth quadric, and $F_{r+1}(X)$ has two isomorphic connected components, denoted by $S_1$ and $S_2$. Each $r$-plane in $X$ is contained in exactly one $(r+1)$-plane in each component of $F_{r+1}(X)$. Hence $F_{r}(X)\simeq \mathbf{P}_{S_1}(\Sigma^*)$. In particular, $\rho(F_{r}(X))=\rho(S_1)+1$. We then compute $\rho(F_{r}(X))$ using the short exact sequence:
$$0\rightarrow \Sym^2\Sigma\rightarrow \Omega_{G}^1|_{F_{r }(X)}\rightarrow \Omega_{F_{r}(X)}^1\rightarrow 0.$$
We again have a resolution of $\cO_{F_{r}(X)}$:
\begin{eqnarray*}
 0\rightarrow\wedge^{\frac{(r+1)(r+2)}{2}}\Sym^2\Sigma\rightarrow\cdots\rightarrow\wedge^{2}\Sym^2\Sigma\rightarrow
 \Sym^2\Sigma\rightarrow\cO_{G}\rightarrow\cO_{F_{r}(X)}\rightarrow 0.
\end{eqnarray*}
By Bott's theorem, we have \begin{eqnarray*}&&h^j(F_{r}(X), \Sym^2\Sigma)=\begin{cases}0& \text{if } j\neq 2\\1& \text{if } j=2\end{cases}\\&& h^j(F_{r}(X),  \Omega_{G}^1|_{F_{r}(X)})=\begin{cases}0& \text{if } j\neq 1\\1& \text{if } j=1\end{cases}\end{eqnarray*}

Therefore, if $n=2r+3$, $\rho(F_r(X))=2$ and if $n=2r+1$, the Picard number of each component of $F_r(X)$ is $1$.

If $n=2r+2$, we can also compute that \begin{eqnarray*}&&h^j(F_{r}(X), \Sym^2\Sigma)=0,\;  \text{for any} \;j  \\&& h^j(F_{r}(X),  \Omega_{G}^1|_{F_{r}(X)})=\begin{cases}0& \text{if } j\neq 1\\1& \text{if } j=1\end{cases}\end{eqnarray*}
hence in this case $\rho(F_r(X))=1$.
\begin{rema}\upshape
A quadric is a homogeneous variety. If $X$ is of dimension $2r+2$ (resp. $2r+1$), $X=G/P_1$ (resp. $G'/P'_1$), where $G$ (resp. $G'$) is a complex simple Lie group of type $D_{r+2}$ (resp. $B_{r+1}$) and $P_1$ (resp. $P'_1$) is a maximal parabolic subgroup of $G$ (resp. $G'$).

Hence the above results can also be found in \cite{LM}, Theorem 4.9. Indeed, Landsberg and Manivel's result says much more.  For instance, their result implies that, if $\dim X=2r+2$, $F_r(X)$ is isomorphic to $G/P_{r+1,r+2}$ and $F_{r+1}(X)=G/P_{r+1}\sqcup G/P_{r+2}$. Hence $\rho(F_r(X))=2$ and the Picard number of each component of $F_{r+1}(X)$ is $1$. Similarly, if $\dim X=2r+1$, $F_r(X)=G'/P'_{r+1}$ and $\rho(F_r(X))=1$.

\end{rema}

\subsection{Case (Q.2)}
In case $(\textrm{Q.2})$, we first consider $X\subset \mathbf{P}(V)=\mathbf{P}^{2r+4}$ the smooth complete intersection of two quadrics. Then $F_r(X)$ is a Fano variety of dimension $2r+2$.
\begin{prop} We have $\dim H^1(F_r(X), \Omega_{F_r(X)}^1)=2r+6$. Hence the Picard number $\rho(F_r(X))$ is $2r+6$.
\end{prop}
The case when $r=1$ is already proved in \cite{B}.
\begin{rema}\upshape It is relatively easy to prove that $\rho(F_r(X))\geq 2r+6$. Since $$\dim H^{2r+2}(X)_{\textrm{prim}}=\dim H^{r+1}(X, \Omega_X^{r+1})_{\textrm{prim}}=2r+5,$$ and the Abel-Jacobi map $H^{2r+2}(X)_{\textrm{prim}}\rightarrow H^2(F_r(X))$ is injective.
\end{rema}
\begin{proof} We assume that $X$ is defined by two quadrics $Q$ and $Q'$.

We claim that \begin{eqnarray}\label{coho1}&&h^j(F_{r}(X), \Sym^2\Sigma)=\begin{cases}0& \text{if } j\neq 2,4\\2r+5& \text{if }j=2\end{cases}\end{eqnarray} \begin{eqnarray}\label{coho2}&&h^j(F_{r}(X), \Omega_G^1|_{F_r(X)})=\begin{cases}1& \text{if } j=1\\2r+5& \text{if }j=2\\0 &\text{if } j\neq 1,2,4 \end{cases}\end{eqnarray}

From \cite[Proposition 2.3.9]{W}, we know that there is a decomposition
\begin{eqnarray}\label{plethym}
\wedge^{m}\Sym^2\Sigma=\oplus_{\lambda}\Gamma^{\lambda}\Sigma,
\end{eqnarray}
where $|\lambda|=2m$ and $\lambda$ ranges over all partitions whose Frobenius notation has the form $\lambda=(\lambda_1-1,\ldots,\lambda_t-t\mid\lambda_1-2,\ldots,\lambda_t-t-1)$, where $t$ is the rank of $\lambda$.

We shall compute\begin{eqnarray*}\label{coho3} H_{k,m}:=H^{k(r+4)}(G, \wedge^m(\Sym^2\Sigma\oplus \Sym^2\Sigma)\otimes \Sym^2\Sigma)\end{eqnarray*}to prove (\ref{coho1}) using the Koszul resolution of $\cO_{F_r(X)}$. More precisely, we will prove that only $H_{1,r+1}$, $H_{1,r+2}$, and $H_{2,2r+4}$ may be nonzero and the natural map $H_{1,r+2}\rightarrow H_{1,r+1}$ is surjective.

If $H_{k,m}\neq 0$, there exists a component $\Gamma^a\Sigma$ of $$\wedge^m(\Sym^2\Sigma\oplus \Sym^2\Sigma)\otimes \Sym^2\Sigma$$ satisfying $a_k\geq r+4+k$ and $a_{k+1}\leq k$. We then assume that $$\Gamma^a\Sigma\subset \Gamma^{\beta}\Sigma\otimes\Gamma^{\gamma}\Sigma\otimes \Sym^2\Sigma,$$ where $$\Gamma^{\beta}\Sigma\subset \wedge^{m_1}\Sym^2\Sigma$$ and $$\Gamma^{\gamma}\Sigma\subset \wedge^{m_2}\Sym^2\Sigma,$$ for some $m_1+m_2=m$.

Since $a_{r+1}\leq \cdots\leq a_{k+1}\leq k$, we have by the Littlewood-Richardson rule that \begin{eqnarray}\label{inequality4}\sum_{i=k+1}^{r+1}(\beta_i+\gamma_i)\leq k(r+1-k),\end{eqnarray} and both $\beta_{k+1}$ and $\gamma_{k+1}$ are $\leq k$.

We denote by $p$ (resp. $q$) the largest integer such that $\beta_p\geq k+1$ (resp. $\gamma_q\geq k+1$). Note that $p, q\leq k$. Moreover, by (\ref{plethym}), we have $$\sum_{i=1}^p\beta_i-(k+1)p\leq \sum_{i=k+1}^{r+1}\beta_i,$$ and $$\sum_{i=1}^q\gamma_i-(k+1)q\leq \sum_{i=k+1}^{r+1}\gamma_i.$$

On the other hand, since $a_1\geq \cdots\geq a_k\geq r+4+k$, we again have by the Littlewood-Richardson rule that
\begin{eqnarray}\label{inequality5}\sum_{i=1}^p\beta_i+\sum_{i=1}^q\gamma_i+k(k-p)+k(k-q)&\geq &\sum_{i=1}^k(\beta_i+\gamma_i) \\\nonumber &\geq& \sum_{i=1}^ka_i-2\\\nonumber&\geq& k(r+4+k)-2.\end{eqnarray}

Combining all the above inequalities, we have $$k(r+3+k)\geq k(r+1+k)+p+q\geq k(r+4+k)-2.$$ Namely, if $H_{k,m}\neq 0$, we should have $k\leq 2$.

If $k=1$ and $H_{1,m}\neq 0$, both $\beta_2$ and $\gamma_2$ are $\leq 1$. Hence there exists $1\leq s\leq r+1$ so that $\beta_1=s+1$ and $\beta_2=\cdots=\beta_s=1>\beta_{s+1}=0$. Similarly, there exists $1\leq t\leq r+1$ so that $\gamma_1=t+1$ and $\gamma_2=\cdots=\gamma_t=1>\gamma_{t+1}=0$. We then conclude that $r+1\leq s+t\leq r+2$. Using Bott's theorem, a direct computation shows that $$H_{1, r+2}\simeq V^{\oplus (r+1)},$$ and $$H_{1, r+1}\simeq (V^{*})^{\oplus r},$$
and the natural map in the Koszul complex induces a surjective map $$H_{1,r+2}\rightarrow H_{1,r+1}.$$

If $k=2$ and $H_{2,m}\neq 0$, equality holds in (\ref{inequality4}) and (\ref{inequality5}). Therefore, $$m=m_1+m_2=\frac{1}{2}(|\beta|+|\gamma|)=2r+4,$$ while $2(r+4)-2r-4=4$.

Therefore, we have finished the proof of (\ref{coho1}).

A similar analysis allows us to prove (\ref{coho2}). Indeed, we can show that
$$H_{t,m}':=H^t(G, \wedge^m(\Sym^2\Sigma\oplus\Sym^2\Sigma)\otimes\Omega_G^1)=0,$$
unless $t=1$ and $m=0$, or $t=r+5$ and $r+2\leq m\leq r+3$, or $t=2r+9$ and $m=2r+5$. Moreover, $\dim H_{1,0}'=1$, and the natural map
$$H_{r+5, r+3}'\simeq V^{\oplus r}\rightarrow H_{r+5, r+2}'\simeq (V^*)^{\oplus (r-1)}$$ is surjective and this proves (\ref{coho2}).

We now finish the proof of the proposition.
By the exact sequence,
\begin{eqnarray*}
0\rightarrow \Sym^2\Sigma \oplus \Sym^2\Sigma\xrightarrow{(\partial Q,\partial Q ')^*}  \Omega_{G}^1|_{F_{r}(X)}\rightarrow \Omega_{F_{r}(X)}^1\rightarrow 0,
\end{eqnarray*}
since $X$ is smooth, the map $(\partial Q,\partial Q ')^*$ induces a surjective map between $$H^2(F_r(X), \Sym^2\Sigma\oplus\Sym^2\Sigma)=V^{\oplus 2}\rightarrow H^2(F_r(X), \Omega_G^1|_{F_r(X)})=V$$
We conclude that \begin{eqnarray*}&&h^1(F_{r}(X), \Omega_{F_r(X)}^1)=2r+6.\end{eqnarray*}
Since $F_r(X)$ is Fano, we have $h^{2,0}(F_r(X))=0$, and this concludes the proof of the proposition.
\end{proof}

Assume now $X\subset \mathbf{P}^{2r+3}$ is a very general complete intersection of two quadrics. We already see in Example \ref{exam1} that $\rho(F_r(X))=1$.
\subsection{Remaining cases}
Some of the remaining cases are classical. For instance,
in case $(\textrm{C.3})$, $X$ is a smooth cubic fourfold. By \cite{BD}, it is known that for $X$ very general, $F_1(X)$ is a very general deformation of $S^{[2]}$ for some polarized $K3$ surface $S$, hence $\rho(F_1(X))=1$.

In the remaining cases $(\textrm{Q}.3)$, $(\textrm{Q}.4)$, $(\textrm{C}.1)$, $(\textrm{C}.2)$, $(\textrm{C}.3)$, $(\textrm{C}.4)$, and $(\textrm{Qr})$,  we claim that





\begin{prop}Under the above assumptions, for each $d_i$, the multiplication \begin{multline*}H^{\delta(n,\underline{d},r)-2}(F_1(X), K_{F_1(X)})\otimes H^0(G, \Sym^{d_i}\Sigma^*)\\\rightarrow H^{\delta(n,\underline{d},r)-2}(F_1(X), \Sym^{d_i}\Sigma^*\otimes K_{F_1(X)})\end{multline*} is surjective.
\end{prop}
We omit the proof since it is again a direct application of Bott's theorem and Lemma \ref{lemma1}. We also notice that in all the above cases, $\delta_{-}(n, \underline{d}, r)=2$. Therefore, by Debarre and Manivel's calculation in \cite{DM}, $$\dim H^1(F_r(X), \Omega^1_G|_{F_r(X)})=1.$$ Then by Lemma \ref{hodge}, we have completed the proof of the main theorem.

\section{The cohomology class of varieties of planes of a cubic fivefold}
\subsection{An intersection formula}\label{se2}
In this section, we will always assume that $Z$ is a general smooth hypersurface of degree $d\geq 3$ in $\mathbf{P}(V)=\mathbf{P}^n$ and the planes contained in $Z$ cover a divisor of $Z$, namely we have $3n-4-\binom{d+2}{2}\geq n-2$. Note that the case of cubic fivefolds satisfies this assumption. We then automatically  have $n-1\geq d$, hence
the lines contained in $Z$ cover the whole of variety $Z$.

We have the following correspondences:
\begin{eqnarray*}\xymatrix{I_1(Z)\ar[r]^{q_1}\ar[d]^{p_1}& Z\\
F_1(Z)}, \quad \xymatrix{I_2(Z)\ar[r]^{q_2}\ar[d]^{p_2}&Z\\
F_2(Z),}
\end{eqnarray*}
where $I_1(Z)$ and $I_2(Z)$ are the incidence varieties. Then $I_1(Z)=\mathbf{P}(\Sigma_1)$ and
$I_2(Z)=\mathbf{P}(\Sigma_2)$, where $\Sigma_1$ and $\Sigma_2$ are respectively the tautological subbundle on $F_1(Z)$ and $F_2(Z)$. We denote respectively by $Q_1$ and $Q_2$ the tautological quotient bundle on $F_1(Z)$ and $F_2(Z)$ and denote by $H_1$ and $H_2$ the respective Pl\"{u}cker polarization. We have $q_1^*\cO_Z(1)=\cO_{p_1}(1)$ and $q_2^*\cO_Z(1)=\cO_{p_2}(1)$ and then set
\begin{eqnarray*}
&&h_1=c_1(\cO_{p_1}(1)),\quad h_2=c_1(\cO_{p_2}(1)), \quad l=c_1(H_1), \quad l'=c_1(H_2);\\
&&c_2=c_2(\Sigma_1^*),\quad c_2' =c_2(\Sigma_2^*),\quad
c_3'=c_3(\Sigma_2^*).
\end{eqnarray*}
By definition, we have the following relations:
\begin{eqnarray*}\label{1}&&h_1^2=h_1 p_1^*l-p_1^*c_2,\nonumber\\&&h_2^3=h_2^2 p_2^*l'-h_2p_2^*c_2'+p_2^*c_3'.\end{eqnarray*} For any
$\alpha\in H^{n-1}(Z, \mathbf{Z})_{\textrm{prim}}$, we may write
\begin{eqnarray}\label{2}
q_1^*\alpha&=&h_1p_1^*\alpha_1+p_1^*\alpha_2,\nonumber\\
q_2^*\alpha&=&h_2^2p_2^*\alpha_1'+h_2p_2^*\alpha_2'+p_2^*\alpha_3',\end{eqnarray}
where $\alpha_i\in H^*(F_1(Z), \mathbf{Z})$ and $\alpha_i'\in H^*(F_2(Z), \mathbf{Z})$. Denote \begin{eqnarray*}\Phi(\alpha)=p_{1*}q_1^*\alpha=\alpha_1\in
H^{n-3}(F_1(Z), \mathbf{Z}),\\ \Psi(\alpha)=p_{2*}q_2^*\alpha=\alpha_1'\in H^{n-5}(F_2(Z),
\mathbf{Z}). \end{eqnarray*}

The following lemma is known (see \cite{BD}).
\begin{lemm}\label{0.1}For any $\alpha, \beta\in H^{n-1}(Z, \mathbf{Z})_{\textrm{prim}}$, we have
$$\big(\Phi(\alpha)\cdot\Phi(\beta)\cdot
l^{n-d}\big)_{F_1(Z)}=-d!\big(\alpha\cdot\beta\big)_Z,$$ and
$c_2\cdot \alpha_1=0.$
\end{lemm}
We then consider the correspondence $C(Z)=\{([\Pi], [L])\in F_2(Z)\times F_1(Z)\mid L\subset \Pi\}$ between
$F_1(Z)$ and $F_2(Z)$:
\begin{eqnarray*}\label{4}
\xymatrix{&C(Z)\ar[dl]_p\ar[dr]^q\\
F_2(Z)&&F_1(Z),}
\end{eqnarray*}
where $p$ and $q$ are natural projections.
Hence $C(Z)=\mathbf{P}_{F_2(Z)}(\Sigma_2^*)$ and there is a tautological sequence of vector bundles on $C(Z)$:
\begin{eqnarray}\label{5}0\rightarrow q^*\Sigma_1\rightarrow p^*\Sigma_2\rightarrow \cO_p(1)\rightarrow 0.
\end{eqnarray}
On the other hand, from the complex
\begin{eqnarray*}
\xymatrix{
0\ar[r]&q^*\Sigma_1\ar@{=}[d]\ar[r]&p^*\Sigma_2\ar@{^{(}->}[d]\ar[r]&\cO_p(1)\ar[r]\ar@{^{(}->}[d]& 0\\
0\ar[r]&q^*\Sigma_1\ar[r]& V\otimes \cO_{C(Z)}\ar[r]& q^*Q_1\ar[r]& 0,}
\end{eqnarray*}
we see that $C(Z)$ is naturally a subscheme of $\mathbf{P}_{F_1(Z)}(Q_1)$: \begin{eqnarray*} \xymatrix{C(Z)\ar@{^{(}->}[r]^{i}\ar[dr]^q&\mathbf{P}_{F_1(Z)}(Q_1)\ar[d]^{\pi}\\
&F_1(Z),}\end{eqnarray*}
and $i^*\cO_{\pi}(1)=\cO_{p}(-1)$.
Moreover, there is a natural short exact sequence on $\mathbf{P}_{F_1(Z)}(Q_1)$,
\begin{eqnarray}\label{6}
\xymatrix{ 0\ar[r]&\pi^*\Sigma_1\ar[r]&\cK\ar[r]&\cO_{\pi}(-1)\ar[r]&0,}
\end{eqnarray}
where $i^*\cK=p^*\Sigma_2$.
Hence we can describe $C(Z)$ to be the zero-locus of a section of
$S^{d-1}\cK^*\otimes\cO_{\pi}(1)$. Indeed, $F$ defines a section of $S^d\cK^*$ which
vanishes on $\pi^*S^d\Sigma_1^*$ and hence defines a section in
$$H^0(\mathbf{P}_{F_1(Z)}(Q_1), S^{d-1}\cK^*\otimes\cO_{\pi}(1))$$ 
and $C(Z)$ is just the zero-locus of this section.

Let $\eta=c_1(\cO_p(1))$. Since $C(Z)=\mathbf{P}_{F_2(Z)}(\Sigma_2^*)$ and by (\ref{5}), we have
\begin{eqnarray}\label{7}\nonumber&&\eta^3=-\eta^2 p^*l'-\eta p^*c_2'-p^*c_3';\\\label{8}&&q^*l=p^*l'+\eta, \quad q^*c_2=p^*c_2' + \eta q^*l=\eta^2+\eta p^*l' +p^*c_2',\end{eqnarray}
and for $\alpha_1\in H^{n-3}(F_1(Z), \mathbf{Z})$ in (\ref{2}), we may write
\begin{eqnarray*}\label{9}q^*\alpha_1=\eta^2 p^*\alpha_3+\eta p^*\alpha_4+p^*\alpha_5.
\end{eqnarray*}

By a direct computation, we have the following:
\begin{lemm}\label{0.2}
For any $\alpha\in H^{n-1}(X, \mathbf{Z})$, let $\alpha_4$ be as
above. Then we have $\alpha_4=\Psi(\alpha)=p_{2*}q_2^*\alpha$.
\end{lemm}
We now have the main result in this subsection.
\begin{prop}\label{0.3}If $Z$ is a general cubic fivefold, for any $\alpha, \beta\in H^5(Z, \mathbf{Z})$, we have
$$\big(\alpha\cdot\beta\big)_Z=\frac{1}{180}\big(\Psi(\alpha)\cdot\Psi(\beta)\cdot l'\big)_{F_2(Z)}.$$
\end{prop}

This proposition comes from a calculation of Voisin in \cite{V2} .

We first assume more generally that we are working on a hypersurface of degree $d$ in $\mathbf{P}^n$ with $3n-4-\binom{d+2}{2}=n-2$.

\begin{itemize}\item[\textbf{Claim 1:}]For any $\alpha$, $\beta \in H^{n-1}(Z)_{\textrm{prim}}$, we define $\alpha_1$, $\beta_1\in H^{n-3}(F_1(Z))$ as in (\ref{2}). There exists a positive integer $N>0$ such that \begin{eqnarray*}(q^*\alpha_1\cdot q^*\beta_1\cdot p^*l')_{C(Z)}=-N(\alpha_1\cdot\beta_1\cdot l^{n-d})_{F_1(Z)}.\end{eqnarray*}
\end{itemize}

We have already seen that $C(Z)\hookrightarrow \mathbf{P}_{F_1(Z)}(Q_1^*)$ is defined by by a section of $S^{d-1}\cK^*\otimes\cO_{\pi}(1)$. We just need to calculate the cohomology class $q_*[p^*l']$ in $F_1(Z)$. By (\ref{6}), this class is a polynomial of $l$ and $c_2$. By Lemma \ref{0.1}, $c_2\alpha_1=0$, hence we are only interested in the coefficient of $l^{n-d}$. We may formally assume that $\Sigma_1^*=H_1\oplus\cO_{F_1(Z)}$. Denote by $c_1(\cO_{\pi}(1))=\varepsilon$. We have seen that $i^*\varepsilon=-\eta$. As $S^{d-1}\cK^*\otimes\cO_{\pi}(1)$ is filtered with successive quotient $$\cO_{\pi}(i)\otimes S^{d-i}\Sigma_1^*$$ for $i=1,\ldots, d$, we have modulo $c_2$, \begin{eqnarray*}c_{\binom{d+1}{2}}(S^{d-1}\cK^*\otimes\cO_{\pi}(1))&=&\prod_{1\leq i\leq j\leq d}(i\varepsilon+(d-j)l)\nonumber\\&=&d!\varepsilon^d\prod_{1\leq i\leq j\leq d-1}(i\varepsilon+(d-j)l).\end{eqnarray*}

By (\ref{8}), $p^*l'=q^*l-\eta=q^*l+\varepsilon$. Therefore
\begin{eqnarray*}q_*[C(Z)]=\pi_*\big(d!(\varepsilon^{d}\prod_{1\leq i\leq j\leq d-1}(i\varepsilon+(d-j)l))\big)\; \textrm{mod $c_2$},
\end{eqnarray*}and
\begin{eqnarray*}q_*[p^*l']=\pi_*\big(d!(\pi^*l+\varepsilon)\varepsilon^{d}\prod_{1\leq i\leq j\leq d-1}(i\varepsilon+(d-j)l)\big)\; \textrm{mod $c_2$}.
\end{eqnarray*}
We define the polynomial in two variables $M(x, y)=\prod_{1\leq i\leq j\leq d-1}(ix+(d-j)y)=\sum_{i=1}^{\binom{d}{2}}\alpha_ix^iy^{\binom{d}{2}-i}$. By symmetry, we have $\alpha_i=\alpha_{\binom{d}{2}-i}$ and it is easy to see that $\alpha_{i-1}<\alpha_i$ for $2i\leq \binom{d}{2}$.

On the other hand, $\pi_*\varepsilon^{n-2+i}=s_i(Q_1^*)=c_i(\Sigma_1)$, hence we have $\pi_*\varepsilon^{n-2}=1$,
$\pi_*\varepsilon^{n-1}=-l$, $\pi_*\varepsilon^{n}=c_2$, and $\pi_*\varepsilon^{n-2+i}=0$, for $i\geq 3$.
We conclude that
\begin{eqnarray*}\label{14}q_*[C(Z)]=d!(\alpha_{n-2-d}-\alpha_{n-1-d})l^{n-d-1}\; \textrm{mod $c_2$},
\end{eqnarray*}
and \begin{eqnarray*}\label{13}q_*[p^*l']&=&d!(\alpha_{n-2-d}-\alpha_{n-1-d}+\alpha_{n-3-d}-\alpha_{n-2-d})l^{n-d}\\
&=&d!(\alpha_{n-3-d}-\alpha_{n-1-d})l^{n-d}\; \textrm{mod $c_2$}.
\end{eqnarray*}
Since $3n-4-\binom{n+2}{2}=n-2$, we have $(n-2-d)+(n-1-d)=\binom{d-1}{2}$, we have $d!(\alpha_{n-2-d}-\alpha_{n-1-d})=0$ and $d!(\alpha_{n-3-d}-\alpha_{n-1-d}):=-N<0$.
Hence $q_*[C(Z)]=0$ \textrm{mod $c_2$} and $$(q^*\alpha_1\cdot q^*\beta_1\cdot p^*l_2)_{C(Z)}=-N(\alpha_1\cdot\beta_1\cdot l^{n-d})_{F_1(Z)}.$$ Hence we have proved Claim 1.

\begin{proofth}We use the notations in the above calculation. Since $$q^*\alpha_1=\eta p^*\alpha_4+p^*\alpha_5,\;\; q^*\beta_1=\eta p^*\beta_4+p^*\beta_5\in H^3(C(Z), \mathbf{Z}),$$ by Lemma \ref{0.2}, we have $$(q^*\alpha_1\cdot q^*\beta_1\cdot p^*l')_{C(Z)}=(\alpha_4\cdot\beta_4\cdot l')_{F_2(Z)}=(\Psi(\alpha)\cdot \Psi(\beta)\cdot l')_{F_2(Z)}.$$

By Lemma \ref{0.1} and Claim 1, we just need to show that $N=30$.
In this case, $M(x, y)=(x+2y)(2x+2y)(2x+y)=2x^3+7x^2y+7xy^2+2y^3$. Therefore $N=d!(\alpha_{n-d-1}-\alpha_{n-d-3})=3!(7-2)=30$.
\end{proofth}
\subsection{Proof of the Theorem 4}
Let $U\subset \mathbf{P}(H^0(\mathbf{P}(V), \cO(3)))$ be the open subset of smooth cubic fivefold which contains $U_0$ as an open subset. There is the universal family of cubics $pr: \cZ\rightarrow U$. We consider the monodromy action
$$\rho: \pi_1(U, 0)\rightarrow \Aut(H^5(Z, \mathbf{Q}), \fI),$$ where $\fI$ is the intersection form on $H^5(Z, \mathbf{Q})$. It is known that the Zariski closure of the image of $\rho$ is the full group $\Sp(H^5(Z, \mathbf{Q}), \fI)$ (see, for instance, \cite{PS}). For any $d$, we denote by $\Hdg^{2d}(JZ):=H^{d,d}(JZ)\cap H^{2d}(JZ, \mathbf{Q})$ the group of Hodge classes of $JZ$. As a corollary of the big monodromy group, we have
 \begin{lemm}\label{6.4.1}$\Hdg^{2d}(JZ)=\mathbf{Q}\langle\Theta^d\rangle$ for all $0\leq d\leq 21$.
\end{lemm}
\begin{proof}We consider the relative intermediate Jacobian $pr: \cJ\rightarrow U$.
Then the Zariski closure of the monodromy action $\pi(U, 0)\rightarrow H^1(JZ, \mathbf{Q})$ is again the full sympletic group. The subspace of invariants in $\wedge^{2d}H^1(JZ, \mathbf{Q})$, with respect to the full sympletic group, is $1$-dimensional and is spanned by $\mathbf{Q}\langle\Theta^d\rangle$.
\end{proof}
\begin{theo}\label{6.4.2}Let $Z$ be a general cubic fivefold. We consider the Abel-Jacobi map $\alpha: F_2(Z)\rightarrow JZ$. Then the cohomology class $[\alpha_*(F_2(Z))]=12[\frac{\Theta^{19}}{19!}]$.
\end{theo}
\begin{proof}We may assume that $Z$ is very general. By Theorem 3 in the introduction, we know that $\rho(F_2(Z))=1$. Hence we have $l'\equiv x\alpha^*\Theta$ for some rational number $x$. By Lemma \ref{6.4.1}, we may also write $[\alpha_*(F_2(Z))]=y[\frac{\Theta^{19}}{19!}]$ for some integer $y$. For any $\alpha, \beta\in H^1(JZ, \mathbf{Z})=H^5(Z, \mathbf{Z})$, we have
\begin{eqnarray*}
\big(\alpha\cdot\beta\cdot \frac{1}{20!}\bigwedge^{20}[\Theta]\big)_{JZ}&=&
\big(\alpha\cdot\beta\big)_Z\\&=&\frac{1}{180}\big(\Psi(\alpha)\cdot\Psi(\beta)\cdot l'\big)_{F_2(Z)}\\
&=&\frac{1}{180}xy\big(\alpha\cdot\beta\cdot \frac{1}{19!}\bigwedge^{20}[\Theta]\big)_{JZ},
\end{eqnarray*}
where the first equality holds by the definition of intermediate jacobian, the second equality holds because of Proposition \ref{0.3}, and the last equality holds by projection formula.
We have $xy=9$. On the other hand, we know from the Remark under Corollary 10 in \cite{IM} that $(l'^2)_{F_2(X)}=2835$. Hence $x^2y=\frac{2835}{21\cdot20}=\frac{27}{4}$. We deduce that $x=\frac{3}{4}$ and $y=12$.
\end{proof}

\begin{rema}\upshape It is not difficult to prove that for any smooth cubic fivefold $Z$, the variety of plane $F_2(Z)$ is always of dimension $2$. Moreover, if $F_2(Z)$ is smooth, the Abel-Jacobi map $\alpha$ is generically injective. Hence $JZ$ has a subvariety of dimension $2$ whose cohomology class is $12[\frac{\Theta^{19}}{19!}]$
\end{rema}

\end{document}